\documentclass[revision]{FPSAC2022}

\title{On the action of the long cycle on the Kazhdan-Lusztig basis}

\author[Gossow \& Yacobi]{Fern Gossow\thanks{\href{mailto:M.Gossow@maths.usyd.edu.au}{M.Gossow@maths.usyd.edu.au}}\addressmark{1}, \and Oded Yacobi\thanks{\href{mailto:oded.yacobi@sydney.edu.au}{oded.yacobi@sydney.edu.au} Oded Yacobi was partially supported by the ARC grant DP180102563}\addressmark{1}}

\address{\addressmark{1}School of Mathematics and Statistics, University of Sydney, Australia}

\received{\today}

\abstract{The complex irreducible representations of the symmetric group carry an important canonical basis called the Kazhdan-Lusztig basis.  Although it is difficult to express how general permutations act on this basis, some distinguished permutations have beautiful descriptions.  In 2010 Rhoades showed that the long cycle $(1,2,...,n)$ acts by the jeu-de-taquin promotion operator in the case when the irreducible representation is indexed by a rectangular partition.  We prove a generalisation of this theorem in two directions: on the one hand we lift the restriction on the shape of the partition, and on the other hand we enlarge the result to the collection of all separable permutations.}


\RequirePackage[vcentermath]{youngtab}
\RequirePackage[backend=bibtex]{biblatex}
\numberwithin{equation}{section}
\newtheorem{Theorem}[equation]{Theorem}
\newtheorem{Proposition}[equation]{Proposition}
\newtheorem{Lemma}[equation]{Lemma}
\newtheorem{Corollary}[equation]{Corollary}

\theoremstyle{definition}

\newtheorem{eg}[equation]{Example}

\newcommand{\bC}{\mathbf{C}}

\newcommand{\bZ}{\mathbf{Z}}

\renewcommand{\phi}{\varphi}

\renewcommand{\tilde}[1]{\widetilde{#1}}

\newcommand\iso\cong
\newcommand\into\hookrightarrow
\newcommand\onto\twoheadrightarrow

\newcommand{\nc}{\newcommand}
\nc{\la}{\lambda}
\nc{\Iso}{\mathsf{Iso}}
\nc{\Irr}{\mathsf{Irr}}
\nc{\Id}{\mathrm{Id}}

\DeclareMathOperator\idx{idx}
\def\Cell{\mathcal{C}}
\def\muu{\overline{\mu}}
\def\pr{\mathrm{pr}}
\def\ev{\mathrm{ev}}
\def\rsk{\rightsquigarrow}

\begin{document}

\maketitle

\section{Main Results}

The complex irreducible representations, aka Specht modules, of the symmetric group $S_n$ are indexed by partitions of $n$:  for such a partition $\lambda\vdash n$ we denote by $S^\lambda$ the associated representation.  These carry a canonical basis called the Kazhdan-Lusztig basis $\{C_T  \mid T \in SYT(\lambda)\}$, which is indexed by the set  of standard Young tableaux of shape $\lambda$.  

Although the Kazhdan-Lusztig basis enjoys many wonderful properties, it is quite complicated to express its transformation under the action of an arbitrary permutation.  Nevertheless, owing to a close connection with the RSK correspondence, it is possible for certain distinguished elements.  

In the mid 1990s Berenstein-Zelevinsky \cite{berenstein} and Stembridge \cite{stembridge} showed independently that the long element of $S_n$ acts on the Kazhdan-Lusztig basis (up to sign) by  Sch\"utzenberger's evacuation operator.  More recently, Rhoades proved in the case of \textit{rectangular} partitions that the long cycle $c=(1,2,\hdots,n)$ acts on the Kazhdan-Lusztig basis (up to sign) by the jeu-de-taquin promotion operator, $\pr$ \cite{rhoades}.  This was the essential new ingredient needed in Rhoades' cyclic sieving phenomenon regarding the action of promotion on rectangular standard Young tableaux.  

Our first main result is a generalisation of Rhoades' Theorem to arbitrary partitions.  Fix a partition $\lambda \vdash n$ and number its removable boxes $1,2,\dots,r$ moving downward.  Define the \emph{index} of $T \in SYT(\lambda)$ to be the said number of the removable box containing $n$, and denote it $\idx(T)$. Choose an ordering of the Kazhdan-Lusztig basis which is weakly increasing along the index.  For $w\in S_n$ let $[w]$ be the matrix of $w:S^\lambda \to S^\lambda$ in this ordered basis.  

\begin{Theorem}\label{thm:main1}
Let $\lambda \vdash n$ be an arbitrary partition.  
Let $[c]=QR$ be the QR decomposition of $[c]$ into the product of an orthogonal matrix $Q$ and an upper triangular matrix $R$.  Then $Q$ is a generalised permutation matrix for $\pr$ with entries $\pm 1$, i.e. for any $T \in SYT(\lambda)$ we have 
$$
QC_T=\pm C_{\pr(T)}.
$$
The sign appearing above depends only on the index of $T$.
\end{Theorem}
Equivalently, this theorem states that $c$ acts on the Kazhdan-Lusztig basis by promotion on the leading term, i.e.
$$
c\cdot C_T = \pm C_{\pr(T)}+ \sum_R a_R C_{\pr(R)},
$$
where the sum is over $R$ with index strictly smaller than the index of $T$, and $a_R \in \bZ$.  In the case of a rectangular partition, the sum is empty and we recover Rhoades' Theorem.  

The key property of rectangular partitions that Rhoades leverages is that the restriction of the corresponding Specht module to $S_{n-1}$ remains irreducible, and the Kazhdan-Lusztig basis is unchanged if the space is regarded as an $S_{n-1}$-module or an $S_n$-module.  It is therefore not surprising that the key step in our argument involves the interaction between the restriction functor and the Kazhdan-Lusztig basis in a more general setting, which may be interesting in its own right.

For $k\geq 0$ define the subspace $S^\lambda_k =\mathrm{span}\{C_T \mid T\in SYT(\lambda) \text{ and } \idx(T)\leq k\}$ of $S^\lambda$. This gives rise to a filtration $0=S^\lambda_0\subset S^\lambda_1 \subset S^\lambda_2 \subset \cdots \subset S^\lambda_r=S^\lambda$, where $r$ is the number of removable boxes of $\lambda$.

\begin{Theorem}\label{thm:branching}Each subspace $S_i^\lambda$ is $S_{n-1}$-invariant, and for every $1 \leq i \leq r$, there is an isomorphism of $S_{n-1}$-modules $S^\lambda_i/S^\lambda_{i-1} \cong S^{\mu_i}$, where $\mu_i \vdash n-1$ is the partition obtained from $\lambda$ by deleting the $i^{\text{th}}$ removable box.  
\end{Theorem}

Note that this theorem provides a new proof that the branching law for symmetric groups is multiplicity-free, as
$$\mathrm{res}^{S_n}_{S_{n-1}}S^\lambda\cong\bigoplus_{i=1}^r S_i^\lambda/S_{i-1}^\lambda\cong\bigoplus_{i=1}^r S^{\mu_i}.$$

Finally, we mention a generalisation of Theorem \ref{thm:main1} to the class of all \emph{separable permutations}, which includes the long element and the long cycle. By definition, separable permutations are those that can be obtained from $e \in S_1$ by repeated applications of direct sum and skew sum.  Alternatively, these are the permutations that avoid the patterns $2413$ and $3142$ \cite{separable}.

Fix $\lambda\vdash n$ an arbitrary partition. In Section \ref{sect:other-actions}, we describe briefly how to associate to each separable permutation $w\in S_n$ a bijection $\varphi$ of $SYT(\lambda)$ such that an analogous result to Theorem \ref{thm:main1} holds. That is, if $[w]:S^\lambda\to S^\lambda$ is the matrix giving the action of $w$ on some given ordering of the KL basis, we have:

\begin{Theorem}\label{thm:generalmain}
Let $w\in S_n$ be any separable permutation, and let $[w]=QR$ be the QR decomposition. Then $Q$ is a (generalised) permutation matrix of $\varphi$.
\end{Theorem}

In specific cases, we can prove this combinatorially but the general case requires results from categorical representation theory and in particular the action of braid groups on triangulated categories \cite{HLLY}.  In our forthcoming work we will describe this connection, and use this to also derive variations of these theorems for arbitrary simply-laced semisimple Lie algebras.

\section{Background}
In this section we briefly recall some basic results we'll need.  We refer the reader to \cite{sagan} for more details on notions from algebraic combinatorics, and to \cite{bjorner} for the necessary background on Kazhdan-Lusztig theory.

\subsection{Combinatorics}\label{sect:combinatorics}
Let $\lambda \vdash n$ be a partition of $n$, which we depict using its associated Young diagram.
Recall that a  box in $\lambda$ is called a \emph{removable box} if its deletion results in another Young diagram. We label the removable boxes of $\lambda$ by $1,2,\dots,r$ starting from the top, and moving down the tableau.  For example, the partition $(6,6,3,1) \vdash 16$ has three removable boxes, with labelling given by:
$$\text{\small$\young(~~~~~~,~~~~~1,~~2,3)$}$$

Write $SYT(\lambda)$ for the set of standard Young tableaux of shape $\lambda$. We define the \emph{index} of $T \in SYT(\lambda)$, denoted $\idx(T)$, to be the label of the removable box containing $n$ as above. If $R$ and $T$ have the same index, then their largest box occupies the same position. When ordering $SYT(\lambda)$ according to index, we can break ties by comparing the index values of the tableaux after removing their largest boxes. Repeating this procedure gives the \emph{total index ordering} on $SYT(\lambda)$.

Recall that the \emph{descent set} $D(T)$ of $T$ is the set of  $j$ in $\{1,2,\dots,n-1\}$ such that $j+1$ appears strictly below $j$ in $T$. 

The \emph{promotion} map $\pr: SYT(\lambda)\to SYT(\lambda)$ is defined as follows: replace the box with value $n$ by a dot, and move this dot north-west by repeatedly swapping it with the box above or to the left. If both boxes exist, swap with the larger. Once the box has reached the north-west corner, increment all values in the tableau and replace the dot with 1.
$$\young(134,25)\to\young(134,2\bullet)\to\young(1\bullet4,23)\to\young(\bullet14,23)\to\young(125,34)$$

The \emph{evacuation} map $\ev_n: SYT(\lambda)\to SYT(\lambda)$ is defined as follows: denote every box in the tableau `fixed' or `unfixed', with all boxes initially starting unfixed. At each step, perform the inverse of the promotion map on the unfixed boxes, and fix the largest remaining box. Repeat until all boxes are fixed. In the example below fixed boxes are in boldface:
$$\young(134,25)\to\young(123,4{{\bf 5}})\to\young(12{{\bf 4}},3{{\bf 5}})\to\young(1{{\bf 3}}{{\bf 4}},2{{\bf 5}})\to\young(1{{\bf 3}}{{\bf 4}},{{\bf 2}}{{\bf 5}})\to
\young({{\bf 1}}{{\bf 3}}{{\bf 4}},{{\bf 2}}{{\bf 5}})$$

Both promotion and evacuation are bijections on $SYT(\lambda)$, and evacuation is in fact an involution. It will also be useful to define the \emph{partial evacuation} maps. For $1\leq k\leq n$, define the involution $\ev_k: SYT(\lambda)\to SYT(\lambda)$ by fixing all but the smallest $k$ boxes, and performing evacuation on the remaining unfixed tableau.
\begin{Lemma}\label{Cor::pr}For every $T\in SYT(\lambda)$, we have $\pr(T)=\ev_n\ev_{n-1}(T)$.\end{Lemma}

The \emph{RSK correspondence} \cite{knuth} gives a bijection between $w \in S_n$ and ordered pairs $(P,Q)$ of tableaux of some shape $\lambda\vdash n$. We will denote the RSK correspondence by $w\rsk(P,Q)$.

\subsection{Representation theory}

The symmetric  group $S_n$ is generated by the simple transpositions $s_j:=(j,j+1)$ for $1\leq j\leq n-1$, satisfying the braid relations and $s_j^2=1$. Given an element $w\in S_n$, we write $l(w)$ for the \emph{length} of $w$ in these generators. 

Given permutations $v,w\in S_n$, we write $v<w$ in the \emph{Bruhat order} if $w$ can be obtained from $v$ by successively swapping two elements in the permutation's one-line notation, increasing the length of the intermediate permutations at each swap. In particular, $v<w$ implies $l(v)<l(w)$. Denote by $D(w)$ the \emph{(left) descent set} of $w$, which consists of the indices $j$ between $1$ and $n-1$ such that $s_jw<w$.

Let $q$ be an indeterminate. The \emph{Hecke algebra (of type A)} $H_n(q)$ is an algebra over $\bZ[q^{\pm 1/2}]$ generated by $\{T_{s_1},\dots,T_{s_{n-1}}\}$ satisfying the braid relations and the quadratic relation $T_{s_j}^2=q+(q-1)T_{s_j}$ for all $1\leq j\leq n-1$. As a $\bZ[q^{\pm 1/2}]$-module, $H_n(q)$ has a basis $\{T_w\}$ indexed by $w\in S_n$. The quadratic relation shows that each $T_{s_j}$, and hence each $T_w$, is invertible.  The \emph{bar involution} $\overline{(\cdot)}$ of the Hecke algebra is the $\bZ$-linear extension of the maps $q\mapsto q^{-1}$ and $T_w\mapsto T_{w^{-1}}^{-1}$.

\begin{Theorem}[Kazhdan-Lusztig Basis, {\cite[Theorem 1.1]{kl}}] There is a unique basis $\{C_w(q)\}$ of $H_n(q)$ indexed by $S_n$, with elements of the form
$$C_w(q)=q^{-l(w)/2}\sum_{v\in S_n}(-1)^{l(v)-l(w)}q^{-l(v)}P_{v,w}(q)T_v,$$
such that: 
\begin{itemize}
\item  $P_{v,w}(q)$ is an integer polynomial, and is non-zero only if $v\leq w$,
\item $P_{v,v}(q)=1$ for all $v\in S_n$,
\item if $v<w$ then $deg(P_{v,w}(q)) \leq (1/2)(l(w)-l(v)-1)$, and
\item $\overline{C_w(q)}=C_w(q)$ for all $w\in S_n$.
\end{itemize}

\end{Theorem}

The $P_{v,w}(q)$ are the well-known \emph{Kazhdan-Lusztig (KL) polynomials}. The degree bound gives us a statistic on pairs of permutations. For $v\leq w$, we write $\muu(v,w)\in\bZ$ for the coefficient of degree $(1/2)(l(w)-l(v)-1)$ in $P_{v,w}(q)$. If $l(w)-l(v)$ is even or $v$ is incomparable with $w$ in the Bruhat order, this coefficient will be zero. We also define $\muu(w,v):=\muu(v,w)$ when $w<v$, so $\muu$ is symmetric.

Under the specialisation $q=1$, we have the relation $T_{s_i}^2=1$, and so $H_n(1)\cong \bZ S_n$. 
Set $C_w=C_w(1)\otimes 1 \in \bZ S_n\otimes_{\bZ}\bC=\bC S_n$, so that $\{C_w\mid w\in S_n\}$ is a basis for $\bC S_n$. Define a binary relation on $S_n$ by setting $v\leq_L w$ if, for some $j$, $C_v$ appears with nonzero coefficient in the expansion of $s_j\cdot C_w$. Taking the transitive closure  gives a preorder on $S_n$, called the \emph{left KL preorder}. By its construction, the span of $\{C_v\mid v\leq_L w\}$ for a fixed $w\in S_n$ will be a submodule of $\bC S_n$. The equivalence classes induced by $\leq_L$ are the \emph{left cells} of $S_n$, and we write $v\sim_L w$ if $v$ and $w$ belong to the same left cell.  

For a left cell $\Cell$ of $S_n$, we write $S^\Cell$ for the linear span of elements $\{C_v\mid v\in\Cell\}$.  
The following defines an irreducible action of $S_n$ on $S^\Cell$ \cite[(6.4)]{bjorner}:
\begin{align}\label{eq:kl_action_v}s_j\cdot C_v=\begin{cases}
-C_v & j\in D(v), \\
C_v+\sum_{w\in\Cell,j\in D(w)}\muu(v,w)C_w&j\not\in D(v).
\end{cases}\end{align}

Left cells in $S_n$ are determined by the $Q$-tableau in the RSK correspondence. That is, $v\sim_L w$ if and only if $v\rsk(P,Q)$ and $w\rsk(P',Q)$ for some $P,P',Q\in SYT(\lambda)$ and $\lambda\vdash n$.
Therefore we can uniquely associate each left cell $\Cell$ of $S_n$ to a tableau $Q$ with $n$ boxes, and write $\Cell_Q$. Appropriately restricting the RSK correspondence gives a bijection between the basis elements $C_v \in \Cell_Q$ and $T \in SYT(\lambda)$, by identifying $C_v$ with the $P$-tableau of $v$. Hence, we can reindex the basis of $S^{\Cell_Q}$ as $\{C_T\mid T\in SYT(\lambda)\}$. We can rewrite the action on the KL basis $\{C_T\}$ in terms of tableaux only:
\begin{enumerate}
\item Let $v\in S_n$ and suppose $v\rsk(P,Q)$. Then $j\in D(v)$ if and only if $j\in D(P)$.
\item For $\lambda\vdash n$ and tableaux $P,T,Q,Q'\in SYT(\lambda)$ we have:
$$\muu((P,Q),(T,Q))=\muu((P,Q'),(T,Q')),$$
where we identify permutations with their images under the RSK correspondence.
\end{enumerate}
From this we obtain that for tableaux $Q,Q'\in SYT(\lambda)$, the respective representations $S^{\Cell_Q}$ and $S^{\Cell_{Q'}}$ are equal (not just isomorphic) up to a reindexing of basis elements.

These results have a number of implications. First, we can define the $\muu$ value between tableaux $P,T\in SYT(\lambda)$ by $\muu(P,T)=\muu((P,Q),(T,Q))$ for any $Q\in SYT(\lambda)$. Furthermore, since $S^{\Cell_Q}$ depends only on the shape of $Q$, we can write $S^\lambda$ for the representation $S^{\Cell_Q}$, where $Q$ is any tableau of shape $\lambda$. Finally, we can write the conditions on the descent set of $v$ in (\ref{eq:kl_action_v}) in terms of the $P$-tableau of $v$ under RSK. We incorporate these results into the following theorem.

\begin{Theorem}[{\cite[Theorem 6.5.3]{bjorner}}]Let $\lambda\vdash n$ be a  partition. The module $S^\lambda$ is irreducible, and has a basis $\{C_T\mid T\in SYT(\lambda)\}$ with the following action:
\begin{align}\label{eq:kl-action}s_j\cdot C_T=\begin{cases}
-C_T & j\in D(T), \\
C_T+\sum_{R\in SYT(\lambda),j\in D(R)}\muu(T,R)C_R&j\not\in D(T).
\end{cases}\end{align}
\end{Theorem}

The irreducible representations of $S_n$ are called  Specht modules.  The above theorem provides a canonical basis for Specht modules, which we refer to as the \emph{KL basis}.

\section{The Branching Rule via the KL Basis}

In this section we prove  Theorem \ref{thm:branching}. Fix $\lambda\vdash n$, and recall the filtration of $S^\lambda$ by the subrepresentations $S_i^\lambda$, each defined as the subspace spanned by $\{C_T\}$ with $\idx(T)\leq i$. We will first show that $S_i^\lambda$ always remains invariant under $S_{n-1}$. 

Let $CSS(\lambda)$ be the \emph{column super-strict} tableau of shape $\lambda$, where the values $1,2,\dots,n$ are placed column-by-column, reading left-to-right. We take $\lambda=(4,3,1)$ as an example:
$$CSS(\lambda)=\young(1468,257,3)$$
Preimages of the RSK correspondence when $Q=CSS(\lambda)$ are easy to describe. We simply write the elements of the $P$-tableau in column order, reading from bottom-to-top.
$$85162734\quad\rsk\quad\left(\,\young(1234,567,8),\,\young(1468,257,3)\,\right)$$

\begin{Lemma}Fix $\lambda\vdash n$ and $1\leq i\leq r$, where $r$ is the number of removable boxes of $\lambda$. Then $S_i^\lambda$ is an $S_{n-1}$-submodule of $S^\lambda$.\end{Lemma}
\begin{proof}If $i=r$, there is nothing to prove. Otherwise, choose arbitrary tableau $T,R$ in $SYT(\lambda)$ with $C_T\in S_i^\lambda$ and $C_R\not\in S_i^\lambda$, or equivalently, $\idx(T)\leq i<\idx(R)$. Instead of working in $S^\lambda$, we work in the module $S^{\Cell_{CSS(\lambda)}}$. To this end, define the unique permutations $v,w\in S_n$ with images $(T,CSS(\lambda))$ and $(R,CSS(\lambda))$ under RSK. We are required to show that $C_w$ does not appear in $s_j\cdot C_v$ (see (\ref{eq:kl-action}) for the action) when $1\leq j\leq n-2$. We consider cases depending on the Bruhat comparibility of $v$ and $w$.

(a) $v$ and $w$ are not Bruhat compatible. Then $\muu(v,w)=0$ by definition.

(b) $w<v$. We have a chain $w=w_1\to\cdots\to w_k=v$ of swaps with $l(w_m)<l(w_{m+1})$. At some point, the swap $w_m\to w_{m+1}$ moves $n$ to the right. But this will decrease the length, which is a contradiction.

(c) $v<w$ and $l(v,w)>1$. Suppose $C_w$ appears in the expansion of $s_j\cdot C_v$. Then $j\in D(w)$ and $j\not\in D(v)$, so $\muu(v,w)=0$ by \cite[Proposition 5.1.9]{bjorner}.\footnote{Björner \cite{bjorner} uses the convention of right descent sets, but the result for left descent sets is analogous.}

(d) $v<w$ and $l(v,w)=1$. Again, suppose $C_w$ appears in the expansion of $s_j\cdot C_v$ for some $s_j\in S_{n-1}$. Then $j\in D(w)$, $j\not\in D(v)$ and $l(v,w)=1$, which implies $w=s_j v$. But $n$ occurs at different positions in the one-line notation of $v$ and $w$, so we must have $j=n-1$. This is again a contradiction, since $s_{n-1}\not\in S_{n-1}$.
\end{proof}

Recall that $\mu_i$ is the partition obtained by deleting the $i^{\text{th}}$ removable box from $\lambda\vdash n$. There is a bijection $d_i:\{T\in SYT(\lambda)\mid \idx(T)=i\}\to SYT(\mu_i)$ which removes the largest box of each tableau $T$. By taking its linear extension, we have an isomorphism of vector spaces $d_i:S_i^\lambda/S_{i-1}^\lambda\to S^{\mu_i}$ for each $1\leq i\leq n-1$. Our goal is to show that this is in fact an isomorphism of $S_{n-1}$-modules. For this, we need to show that the deletion map preserves the function $\muu$.

We again define a specific recording tableau and work with the permutations explicitly. For a partition $\lambda\vdash n$ with $r$ removable boxes and $1\leq i\leq r$, define the tableau $CSS_i(\lambda)$ by the following procedure:

\begin{itemize}
\item Create a tableau of shape $\lambda$, and place the value $n$ in removable box $i$.
\item Suppose $n$ is placed in the $k^\text{th}$ row.  For $1\leq m\leq k-1$, place the value $n-m$ at the end of row $k-m$.
\item Place the remaining values $1,\dots,n-k$ into the tableau column by column.
\end{itemize}

For example, if $\lambda=(4,3,1)$ and $i=2$ we have:
$$CSS_i(\lambda)=\young(1467,258,3)$$
The proof that the $\muu$ function is preserved under the deletion map was originally given by Rhoades in the case of rectangular partitions. The following Lemma appears implicitly in the proof:

\begin{Lemma}[\cite{rhoades}, Lemma 3.1]Suppose $u=u_1\cdots u_{n-1}$ and $v=v_1\cdots v_{n-1}$ are elements of $S_{n-1}$, written in their one-line notation. For a fixed $0\leq k\leq n-1$, define the permutations
\begin{align*}u'&=u_1\cdots u_knu_{k+1}\cdots u_{n-1},\\ 
v'&=v_1\cdots v_knv_{k+1}\cdots v_{n-1}\end{align*}
in $S_n$. Then we have an equality of $\muu$ values $\muu(u,v)=\muu(u',v')$.\end{Lemma}

\begin{Proposition}\label{prop::d_mu}Fix $\lambda\vdash n$, and choose tableaux $T,R\in SYT(\lambda)$, both with index $i$. Then
\[\muu(T,R)=\muu(d_i(T),d_i(R))\]
\end{Proposition}
\begin{proof}Suppose the removable box $i$ of $\lambda$ occurs in row $k$. Define the following permutations by their images under the RSK correspondence, using the tableau $CSS_i(\lambda)$.
\[\begin{matrix}u\rsk(T,CSS_i(\lambda)) & \quad & \tilde{u}\rsk(d_i(T),d_i(CSS_i(\lambda))) \\
v\rsk(R,CSS_i(\lambda)) & \quad & \tilde{v}\rsk(d_i(R),d_i(CSS_i(\lambda)))\end{matrix}\]
By the above Lemma, it remains to show that $u=\tilde{u}_1\cdots \tilde{u}_{n-k}n\tilde{u}_{n-k+1}\cdots \tilde{u}_{n-1}$ and the result for $v$ follows \emph{mutatis mutandis}.

The box containing the value $n-k+1$ in $CSS_i(\lambda)$ will be in the top row. Hence, $\tilde{u}_{n-k+1}$ is inserted directly into the top row when performing RSK on $\tilde{u}$. In $u$, we instead insert $n$, which will also insert into the top row. Inserting $\tilde{u}_{n-k+1}$ will then push $n$ into the next row. Each following insertion for $u$ mimics $\tilde{u}$, except the $n$ box is constantly pushed down the right-hand side of the insertion tableau. One can verify that this gives the correct tableaux-pair for $u$, with an additional $n$-box in row $k$ for each.
\end{proof}

For a tableau $T\in SYT(\lambda)$ with index $i$, write $[C_T]$ for the basis element $C_T+S_{i-1}^\lambda$ of $S_i^\lambda/S_{i-1}^\lambda$, where $1\leq i\leq n-1$ is the unique value such that $C_T \in S_i^\lambda$, but $C_T \notin S_{i-1}^\lambda$.

By (\ref{eq:kl-action}), for a tableau $T$ with index $i$ we have:
$$s_j\cdot [C_T]=\begin{cases}
-[C_T] & j\in D(T) \\
[C_T]+\sum_{R}\muu(T,R)[C_R]&j\not\in D(T)\end{cases}
$$
In the second case, we sum over $R\in SYT(\lambda)$ with $j\in D(R)$ and $\idx(R)=\idx(T)$. Then:
$$d_is_j\cdot [C_T]=\begin{cases}
-C_{d_i(T)} & j\in D(T) \\
C_{d_i(T)} + \sum_R\muu(T,R)C_{d_i(R)} & j\not\in D(T)
\end{cases}$$
with the sum in the second case as before. If $1\leq j\leq n-2$, then $j\in D(T)$ if and only if $j\in D(d_i(T))$, and likewise for $R$. Hence:
$$d_is_j\cdot [C_T]=\begin{cases}
-C_{d_i(T)} & j\in D(d_i(T)) \\
C_{d_i(T)} + \sum_R\muu(d_i(T),d_i(R))C_{d_i(R)} & j\not\in D(d_i(T))
\end{cases}$$
This is the action of $s_j$ on $C_{d_i(T)}$ in $S^{\mu_i}$. Thus, $s_j$ and $d_i$ commute on every basis element $[C_T]$ of $S_i^\lambda/S_{i-1}^\lambda$ for every generator $s_j$ of $S_{n-1}$. This completes the proof of Theorem \ref{thm:branching}. 

\section{The long cycle action on the KL basis}
In this section we prove Theorem \ref{thm:main1}. Let $w_n\in S_n$ be the long element, and $w_{n-1}$ be the image of the long element of $S_{n-1}$ under the standard embedding into $S_n$. Our goal is to relate the identity $c=w_nw_{n-1}$ to $\pr=\ev_n\ev_{n-1}$ from Lemma \ref{Cor::pr}.

Note that $w_{n-1}\in S_{n-1}$ will act on a tableau of size $n-1$ by $\ev_{n-1}$ (up to sign) \cite{stembridge}. Moreover, $d_i$ commutes with $w_{n-1}$ on $S_i^\lambda/S_{i-1}^\lambda$. Therefore
\begin{align*}w_{n-1}\cdot [C_T]&=d_i^{-1}w_{n-1}d_i\cdot [C_T]\\
&=d_i^{-1}w_{n-1} C_{d_i(T)}\\
&=\pm d_i^{-1} C_{\ev_{n-1}d_i(T)}=\pm [C_{\ev_{n-1}(T)}].\end{align*}
Note that the sign depends only on the shape of $d_i(T)$. Since the index of $T$ matches the index of $\ev_{n-1}(T)$, we obtain:

\begin{Lemma}Fix a partition $\lambda\vdash n$. The action of $w_{n-1}$ on the KL basis $\{C_T\}$ of $S^\lambda$ is
$$w_{n-1}\cdot C_T=\pm C_{\ev_{n-1}(T)}+\sum_R a_RC_R,$$
with the sum over $R\in SYT(\lambda)$ satisfying $\idx(R)<\idx(T)$, and $a_R\in\bZ$ are unknown coefficients. Furthermore, the sign of $C_{\ev_{n-1}(T)}$ depends only on $\idx(T)$.\end{Lemma}

Theorem \ref{thm:main1} follows immediately from:

\begin{Theorem}\label{thm:mainthmredux}
Fix a partition $\lambda\vdash n$. The action of $c$ on the KL basis $\{C_T\}$ of $S^\lambda$ is
$$c\cdot C_T=\pm C_{\pr(T)}+\sum_R b_RC_{\pr(R)},$$
with the sum over $R\in SYT(\lambda)$ satisfying $\idx(R)<\idx(T)$, and $b_R\in\bZ$ are unknown coefficients. Furthermore, the sign of $C_{\pr(T)}$ depends only on $\idx(T)$.\end{Theorem}
\begin{proof}Reindexing the sum from the Lemma with $R\mapsto{\ev_{n-1}(R)}$,
$$w_{n-1}\cdot C_T=\pm C_{\ev_{n-1}(T)}+\sum_R a_{\ev_{n-1}(R)}C_{\ev_{n-1}(R)}.$$
We set $b_R:=a_{\ev_{n-1}(R)}$ and apply the long element to both sides:
$$w_nw_{n-1}\cdot C_T=\pm C_{\ev_n\ev_{n-1}}+\sum_R \pm b_RC_{\ev_n\ev_{n-1}(R)}.$$
Replacing $w_nw_{n-1}$ with $c$ and $\ev_n\ev_{n-1}$ with $\pr$ gives the result required.
\end{proof}

\begin{eg}Fix $\lambda=(3,1,1)$. We arrange $SYT(\lambda)$ using the total index ordering:
\begin{align*}
\young(145,2,3) \prec\; \young(135,2,4) \prec\; \young(125,3,4) \prec\; \young(134,2,5) \prec\; \young(124,3,5) \prec\; \young(123,4,5)
\end{align*}
Using \texttt{MAGMA} \cite{magma} we compute $[c]$ with respect to this ordered basis and calculate its QR decomposition:
\[[c]=\text{\footnotesize$\begin{pmatrix}
0 & 0 & 0 & 1 & 0 & 0 \\
0 & 0 & 0 & 0 & 1 & 0 \\
1 & 0 & 0 & -1 & 1 & 0 \\
0 & 0 & 0 & 0 & 0 & 1 \\
0 & 1 & 0 & -1 & 0 & 1 \\
0 & 0 & 1 & 0 & -1 & 1\end{pmatrix}$}=\text{\footnotesize$\begin{pmatrix}
0 & 0 & 0 & 1 & 0 & 0 \\
0 & 0 & 0 & 0 & 1 & 0 \\
1 & 0 & 0 & 0 & 0 & 0 \\
0 & 0 & 0 & 0 & 0 & 1 \\
0 & 1 & 0 & 0 & 0 & 0 \\
0 & 0 & 1 & 0 & 0 & 0\end{pmatrix}
\begin{pmatrix}
1 & 0 & 0 & -1 & 1 & 0 \\
0 & 1 & 0 & -1 & 0 & 1 \\
0 & 0 & 1 & 0 & -1 & 1 \\
0 & 0 & 0 & 1 & 0 & 0 \\
0 & 0 & 0 & 0 & 1 & 0 \\
0 & 0 & 0 & 0 & 0 & 1 \end{pmatrix}
$}\]
One can verify that $Q$ is precisely the permutation matrix $[\pr]$.  
\end{eg}

\begin{Corollary}Let $T\in SYT(\lambda)$ and suppose $T$ has index $1$.  Then
$c \cdot C_T=\pm C_{\pr(T)}$,
with the sign depending only on $\lambda$. In particular, when $\lambda$ is a rectangular partition, then $c$ acts by the promotion operator on the KL basis of $S^\lambda$ up to sign.\end{Corollary}

\section{Other permutations acting on the KL basis}\label{sect:other-actions}
In this section we discuss Theorem \ref{thm:generalmain}.  Details will appear in forthcoming work.  

Fix a partition $\lambda\vdash n$. Let $I$ denote the set of generators $\{s_1,\dots,s_{n-1}\}$ of $S_n$. For a non-empty subset $J\subseteq I$, let $S_{J}\leq S_n$ denote the associated parabolic subgroup, and $w_{J}$ its longest element. To each such subset, we can associate a filtration $0=S_0^\lambda\subset S_1^\lambda\subset\cdots\subset S_p^\lambda=S^\lambda$ of the Specht module such that each subspace $S_j^\lambda$ is $S_J$-invariant and spanned by KL basis elements,  and the quotient representation $S_j^\lambda/S_{j-1}^\lambda$ is isomorphic to a simple $S_{J}$-module.

This filtration induces a total preorder $\prec_{J}$ on $SYT(\lambda)$, with $T\prec_{J} R$ when $C_T$ appears strictly before $C_R$ in this filtration. Using partial evacuation operators, we  also associate a bijection $\varphi_{J}$ on $SYT(\lambda)$ such that the action of $w_{J}$ on the KL basis is given by
$$w_{J}\cdot C_T=\pm C_{\phi_{J}(T)}+\sum_{R\prec_{J} T}a_RC_{\phi_{J}(R)}$$
for some constants $a_R\in\bZ$, and a sign which depends only on the equivalence class of $T$ under $\prec_{J}$. 

We can extend this result in a natural way to \emph{chains} of subsets of $J$. Call a permutation \emph{descending} if it is of the form $w_{J_\bullet}=w_{J_k}\cdots w_{J_1}$ for some increasing chain $J_\bullet$ of subsets $J_1\subset\cdots\subset J_k\subseteq J$. It turns out that a permutation is descending if and only if it is separable.\footnote{This observation was communicated to the authors by Joel Gibson, who discovered this fact by enumerating the descending permutations computationally in MAGMA.} Each descending permutation has an associated bijection $\varphi_{J_\bullet}=\varphi_{J_k}\cdots\varphi_{J_1}$ and a total preorder $\prec_{J_\bullet}$ on $SYT(\lambda)$. The statement is as before, with the action of $w_{J_\bullet}$ on the KL basis of $S^\lambda$ given by
\begin{equation}\label{eq:general_action}w_{J_\bullet}\cdot C_T=\pm C_{\phi_{J_\bullet}(T)}+\sum_{R\prec_{J_\bullet} T}b_RC_{\phi_{J_\bullet}(R)}\end{equation}
for some constants $b_R\in\bZ$, and a sign which depends only on the equivalence class of $T$ under $\prec_{J_\bullet}$. Theorem \ref{thm:generalmain} follows from this formula.

When $J\subseteq I$ is a \emph{connected} subset $\{s_{a+1},\dots,s_b\}$ for $0\leq a\leq b<n$, then $S_{J}$ is isomorphic to the symmetric group $S_{b-a}$. Moreover, $w_{J}=w_{b+1}w_{b-a}w_{b+1}$, where $w_k\in S_n$ is the permutation reversing $\{1,\dots,k\}$. Likewise, $\varphi_{J}=\ev_{b+1}\ev_{b-a}\ev_{b+1}$, where $\ev_k$ is the partial Schützenberger involution from Section \ref{sect:combinatorics}. Finally, the preorder $\prec_{J_i}$ can be defined explicitly by: $R\prec_{J} T$ if and only if $\ev_{a}\ev_n(R)$ precedes $\ev_{a}\ev_n(T)$ in the total index ordering.

In the case of chains of \emph{connected} subsets of $J$, equation (\ref{eq:general_action}) can be proved combinatorially. However, in the case of general chains, we use results from categorical representation theory, which will appear in later work by the authors. This will also prove analogues of these results in other types concerning the action of an ADE Weyl group on the Lusztig's dual canonical basis in the zero weight space of an irreducible representation of the corresponding Lie algebra. 

Finally, we note that the result in (\ref{eq:general_action}) does not hold  for arbitrary permutations. For example, take the non-separable permutation $w=2413\in S_4$ and $\lambda=(3,1)$. Choose any of the $3!$ orderings of the KL basis for $S^\lambda$, and compute the $QR$ decomposition of $[w]:S^\lambda\to S^\lambda$ with respect to this basis. One can verify that $Q$ is not a generalised permutation matrix under any of these orderings.

\acknowledgements{The results in Sections 1-4 are part of the first author's University of Sydney Honours Thesis (2021), supervised by the second author.  The authors thank the anonymous referees for helpful comments, and Joel Gibson, Ed Heng, Tony Licata and Eloise Little for useful conversations that improved this work.}

\printbibliography

\end{document}